\documentclass[12pt]{amsart}

\usepackage{amsfonts, amsthm, amsmath, amssymb, url, xspace, enumerate}
\usepackage{graphics, graphicx, epsf, color}
\usepackage{caption, subcaption}
\usepackage{lineno}
\usepackage{fullpage}

\newtheorem{thm}{Theorem}[section]

\newtheorem{lemma}[thm]{Lemma}

\newtheorem{cor}[thm]{Corollary}

\newtheorem{property}[thm]{Property}

\theoremstyle{definition}

\theoremstyle{remark}

\newtheorem{remark}{Remark}

\newcommand{\comment}[1]{}
\def\ignore#1{{ }}

\newcommand{\N}{\mathbb N}
\newcommand{\Z}{\mathbb Z}

\newcommand{\T}{\mathcal{T}}

\begin{document}


\author{H.A. Kierstead}
\address{Department of Mathematics and Statistics, Arizona State University,
Tempe, AZ 85287, USA. }
\thanks{Research
of this author is supported in part by NSA grant H98230-12-1-0212.}
\email{kierstead@asu.edu.}

\author{A.V. Kostochka}
\address{Department of Mathematics \\ University of Illinois \\ Urbana, IL 61801, USA\\ and Sobolev Institute of Mathematics\\ Novosibirsk, Russia}
\thanks{Research of this author is supported in part by NSF grant  DMS-1266016 and  by grants 15-01-05867  and 16-01-00499
 of the Russian Foundation for Basic Research. }
\email{kostochk@math.uiuc.edu}

\author{A. M\lowercase{c}Convey}
\address{Department of Mathematics \\ University of Illinois \\ Urbana, IL 61801\\ USA}
\thanks{This author gratefully acknowledges support from the Campus Research Board, University of Illinois.}
\email[Corresponding author]{mcconve2@illinois.edu}

\title{Strengthening theorems of Dirac and Erd\H{o}s on disjoint cycles}

\begin{abstract} 
Let $k \geq 3$ be an integer, $H_{k}(G)$ be the set of vertices of degree at least $2k$
in a graph $G$, and $L_{k}(G)$ be the set of vertices of degree at most $2k-2$ in $G$.
  In 1963, Dirac and Erd\H{o}s proved that $G$ contains $k$ (vertex-)disjoint cycles 
  whenever $|H_{k}(G)| - |L_{k}(G)| \geq k^{2} + 2k - 4$.  
 The main result of this paper is that for $k \geq 2$, every
  graph $G$ with $|V(G)| \geq 3k$ containing at most  $t$ disjoint triangles and with $|H_{k}(G)| - |L_{k}(G)| \geq 2k + t$
   contains $k$ disjoint cycles.  This yields that if $k \geq 2$ and $|H_{k}(G)| - |L_{k}(G)| \geq 3k$, then $G$ contains $k$ disjoint cycles.  
   This generalizes the Corr\'{a}di-Hajnal Theorem, which states that every graph $G$  with $H_{k}(G) = V(G)$ and $|H_{k}(G)| \geq 3k$
   contains $k$ disjoint cycles.
\end{abstract}

\maketitle

{\small{Mathematics Subject Classification: 05C35, 05C70, 05C10.}}{\small \par}

{\small{Keywords: Disjoint Cycles, Disjoint Triangles, Minimum Degree, Planar Graphs.\\}}{\small \par}
 
 \section{Introduction}
For a graph $G$, let $|G| = |V(G)|$, $\| G \| = |E(G)|$, and $\delta(G)$ be the minimum degree of a vertex in $G$. 
The complement of $G$ is denoted by $\overline G$.  The join $G\vee G'$ of two graphs is $G \cup G'\cup \{xx':x\in V(G)~\mbox{and}~x'\in V(G')\}$. 
 Let $SK_{m}$ denote the graph obtained by subdividing one edge of the complete $m$-vertex graph $K_{m}$. 
 For a positive integer $k$, define $H_{k}(G)$ to be the subset of vertices with degree at least $2k$ and $L_{k}(G)$ to be the subset of vertices of degree at most $2k - 2$.

Resolving a conjecture of Erd\H{o}s,  Corr\'{a}di and Hajnal~\cite{C-H} proved the following theorem.

\begin{thm}\cite{C-H}\label{thm:C-H}
Let $G$ be a graph and $k$ a positive integer.  If $|G| \geq 3k$ and $\delta(G) \geq 2k$, then $G$ contains $k$ disjoint cycles.
\end{thm}

Since each cycle has at least $3$ vertices, the condition $|G| \geq 3k$ is necessary.  The condition $\delta(G) \geq 2k$ is also sharp, 
as witnessed by the graph $G_{n,k}=\overline{K}_{n - 2k +1} \vee K_{2k-1}$ for $n \geq 3k$.  Indeed, any cycle in $G_{n,k}$ 
must contain at least two vertices from the $K_{2k-1}$ 
 and therefore $G_{n,k}$ contains at most $k-1$ disjoint cycles.  
 
 Theorem~\ref{thm:C-H} prompted a series of refinements and extensions for both undirected graphs
 (see, e.g.,~\cite{D-E, Di, HSz, Enomoto, Wang, KK-Ore, CFKS,  KK-refCH, K-K-Y, KKMY}) and directed graphs (see, e.g.,~\cite{WangDir, CKM, CDKM, T}).
In particular,  Dirac and Erd\H{o}s \cite{D-E}  proved in 1963 the following theorem.

\begin{thm}\cite{D-E}\label{thm:D-E}
Let $k \geq 3$  be an integer and $G$ be a graph  with $|H_{k}(G) | - |L_{k}(G) | \geq k^{2} + 2k - 4$.  Then $G$ contains $k$ disjoint cycles.
\end{thm}

The bound $k^2+2k-4$ is not best possible.  Dirac and Erd\H{o}s provided the following examples 
of a graph $G$ without $k$ disjoint cycles such that $|H_{k}(G)| - |L_{k}(G)| = 2k - 1$.  Let $n\geq 3k$ be odd.
Let $V(G) = X \cup Y \cup Z$ with $|X| = 2k - 1$ and $|Y| = |Z| = \frac{n-2k+1}{2} $.  Let the set of edges
of $G$ consist of a perfect matching connecting $Y$ with $Z$, 
 all edges between $X$ and $Y$, and all  edges inside $X$.  
Then, $H_{k}(G) = X \cup Y$ and $L_{k}(G) = Z$, but every cycle must contain at least two vertices of $X$.

 Dirac and Erd\H os also proved that for a planar graph $G$ weaker restrictions on the difference $|H_{k}(G)| - |L_{k}(G)|$ 
 provide that $G$ contains $k$ cycles 

\begin{thm}\cite{D-E}\label{thm:D-Eplanar}
Let $k \geq 3$ be an integer and $G$ be a planar graph such that $|H_{k}(G)| - |L_{k}(G)| \geq 5k - 7$.  Then $G$ contains $k$ disjoint cycles.
\end{thm}

The main result of this paper is the following theorem.

\begin{thm}\label{thm:2k+t}
Let $k \geq 2$ be an integer and $G$ be a graph such that $|G| \geq 3k$.  
Let $t$ be the maximum number of disjoint triangles contained in $G$.  If
\[ |H_{k}(G) | - |L_{k}(G) | \geq 2k + t,\]
then $G$ contains $k$ disjoint cycles.
\end{thm}

Theorem~\ref{thm:2k+t} is sharp, as witnessed by the graph $SK_{3k-1}$.  
Let $u$ be the newly created vertex (of degree $2$) and observe that $L_{k}(SK_{3k-1}) = \{ u \}$, $|H_{k}(SK_{3k-1})| = 3k - 1$, and $G$ contains 
$k-1$ disjoint triangles. 
However since $|SK_{3k-1}|=3k$, any set of $k$ disjoint cycles must partition $V(SK_{3k-1})$ into triangles, but the vertex $u$ is not contained in any triangle.  
As above, $3k$ vertices are  necessary for the existence of $k$ cycles.  
However, if the bound on $|H_k(G)|-|L_k(G)|$ in Theorem~\ref{thm:2k+t} is slightly strengthened, then the condition $|G| \geq 3k$ holds 
automatically.

\begin{cor}\label{cor:2k+t+1}
Let $k \geq 2$ be an integer and $G$ be a graph.  Let $t$ be the maximum number of disjoint triangles contained 
in $G$.  If
\[ |H_{k}(G) | - |L_{k}(G) | \geq 2k + t + 1,\]
then $G$ contains $k$ disjoint cycles.
\end{cor}

Corollary~\ref{cor:2k+t+1} is sharp since $K_{3k-1}$  contains only $k-1$ disjoint cycles and $|H_{k}(K_{3k-1})| - |L_{k}(K_{3k-1})| = 3k-1$.  
Further, Corollary~\ref{cor:2k+t+1} implies the following stronger version of Theorem~\ref{thm:D-E}.

\begin{cor}\label{cor:3k}
Let $k \geq 2$ be an integer and $G$ be a graph with $|H_{k} (G)| - |L_{k}(G)| \geq 3k$.  Then $G$ contains $k$ disjoint cycles.
\end{cor}

Observe that the special case  $H_{k}(G) = V(G)$ of Corollary~\ref{cor:3k} is equivalent to Theorem~\ref{thm:C-H} for $k \geq 2$. 
Using the techniques of Theorem~\ref{thm:2k+t}, we will also prove the following stronger version of Theorem~\ref{thm:D-Eplanar}.

\begin{thm}\label{thm:planar}
Let $k \geq 2$ be an integer and $G$ be a planar graph.  If
\[ |H_{k}(G)| - |L_{k}(G)| \geq  2k,\]
then $G$ contains $k$ disjoint cycles.
\end{thm}

The condition that $G$ be planar is necessary.  Indeed, consider the non-planar graph $SK_{5}$.  If $u$ is the newly created vertex, then $H_{2}(SK_{5}) = V(SK_{5}) - u$ and $L_{2}(SK_{5}) = \{u\}$, but $SK_{5}$ does not have two disjoint cycles.  The bound $2k$ in Theorem~\ref{thm:planar} is sharp (see, e.g. $K_{5} - e$ for $k=2$), however only for small values of $k$. Since the average degree of every planar graph is less than $6$, for $k\geq 5$ much weaker
restrictions provide existence of $k$ disjoint cycles in planar graphs.

Finally, we prove that the bound $2k$ is also sufficient if the graph $G$ contains no two disjoint triangles.

\begin{thm}\label{thm:1tri}
Let $k \geq 3$ be an integer and $G$ be a graph such that $G$ does not contain two disjoint triangles.  If
\[ |H_{k}(G)| - |L_{k}(G)| \geq 2k,\]
then $G$ contains $k$ disjoint cycles.
\end{thm}

Our proofs  are based on the approach and ideas of Dirac and Erd\H{o}s \cite{D-E}. We also heavily use an extension of Theorem~\ref{thm:C-H}
from~\cite{K-K-Y} (Theorem~\ref{thm:K-K-Y} below).

The remainder of this paper is organized as follows.  The next section outlines the notation that we will use throughout the paper, and introduces some tools to be used in the proof.  In Section~\ref{sec:basecase} we will prove several lemmas for the base case, and in Section~\ref{sec:2k+t} we prove the main result.  
In Sections~\ref{sec:2k+t+1},~\ref{sec:planar} and~\ref{sec:1tri}, we use Theorem~\ref{thm:2k+t} to prove Corollary~\ref{cor:2k+t+1},  Theorem~\ref{thm:planar}, and Theorem~\ref{thm:1tri}, respectively.

\section{Notation and Tools}\label{sec:notation}

For disjoint sets $U, U' \subseteq V(G)$, we write $\| U, U' \|$ for the number of edges from $U$ to $U'$.  
If $U = \{ u \}$, then we will write $\| u, U' \|$ instead of $\| \{ u \}, U' \|$.
For $x\in V(G)$, $N_{G}(x)$ is the set of vertices adjacent to $x$ in $G$  and $d_{G}(x) = |N_{G}(x)|$.  
When the choice of $G$ is clear, we  simplify the notation to $N(x)$ and $d(x)$, respectively.  
For  an edge $xy\in E(G)$, $G\diagup xy$ denotes the graph obtained from $G$ by contracting $xy$, and
$v_{xy}$ denotes the vertex resulting from contracting $xy$.
By $\alpha(G)$ we denote the {\em independence number} of $G$.

We say that $x, y, z \in V(G)$  {\em form a triangle} $T = xyz$ in $G$ if $G[\{x,y,z\}]$ is a triangle.  
We say $v \in T$, if $v \in \{x, y, z \}$.  A set $\T$ of triangles is a set of subgraphs of $G$ such that each subgraph is a triangle and all the triangles are disjoint. 
For a set $\mathcal S$ of graphs, let $V(\mathcal S)=\bigcup \{V(S):S \in \mathcal S\}$.  For a  graph $G$, we write $t_{G}$ for the maximum number of disjoint triangles contained in $G$.  
\begin{equation}\label{pb1}
\parbox{6in}{\em When the graph $G$ is clear from context, we will use $t$ instead of $t_{G}$.  Similarly, when the integer $k$ is also clear, we will
 use $H$ and $L$ for $H_{ k}(G)$ and $L_{k}(G)$, respectively.  The sizes of $H$ and $L$ will be denoted by $h$ and $\ell$, respectively.}
\end{equation}

As shown in \cite{K-K-Y}, if a graph $G$ with $|G| \geq 3k$ and $\delta(G) \geq 2k - 1$ does not contain a large independent set, then with two exceptions, 
$G$ contains $k$ disjoint cycles:

\begin{thm}\cite{K-K-Y}\label{thm:K-K-Y}
Let $k \geq 2$.  Let $G$ be a graph with $|G| \geq 3k$ and $\delta(G) \geq 2k -1$ such that $G$ does not contain $k$ disjoint cycles.  Then, either
\begin{enumerate}
\item $\alpha(G) \geq |G| - 2k + 1$, or
\item $k$ is odd and $G = 2K_{k} \vee \overline{K_{k}}$, or
\item $k =2$ and $G$ is a wheel.
\end{enumerate}
\end{thm}

We will use the following corollary of Theorem~\ref{thm:K-K-Y} throughout the paper.

\begin{cor}\label{cor:K-K-Y}
Let $k \geq 2$ be an integer and $G$ be a graph with $|G| \geq 3k$.  If $|H|  \geq 2k$ and $\delta(G) \geq 2k - 1$ (i.e. $L = \emptyset$),
 then $G$ contains $k$ disjoint cycles.
\end{cor}

\begin{proof}
First, if $G = 2K_{k} \vee \overline{K_{k}}$, then $|H| = k$, a contradiction.  Next, if $\alpha(G) \geq |G| - 2k + 1$, 
then let $U$ be an independent set of size $|G| - 2k + 1$.  For each $u \in U$, $d(u) \leq 2k - 1$, so 
$H \subseteq V(G) \setminus U$ and $|H_{k}(G)| \leq 2k - 1$. Finally, if $k = 2$, then $G$ is not a wheel, since the wheel has only one vertex of degree at least $4$.  Therefore, by Theorem~\ref{thm:K-K-Y}, $G$ contains $k$ disjoint cycles.
\end{proof}

Call a graph $G$ minimal if among graphs with a certain property, $|G|$ is minimal, and subject to this, $\|G\|$ is minimal.  Dirac and Erd\H{o}s \cite{D-E} observed the following.

\begin{property}\label{prop:minimal}
Let $k \geq 2$ be an integer and $f: \N \to \Z$ a function.  
Suppose $G$ is  minimal among  graphs without $k$ disjoint cycles satisfying $|H| - |L| \geq f(k)$.  Then,
\begin{enumerate}
\item $\delta(G) \geq 2$, and 
\item if $uv \in E(G)$, then either $d(u) \in \{ 2k-1, 2k \}$ or $d(v) \in \{ 2k-1, 2k \}$.
\end{enumerate}
\end{property}

Indeed, if such a graph $G$ contained a vertex $v$ with $d(v) \leq 1$, then $G - v$ is a smaller counterexample.  Similarly, if \emph{(2)} does not hold, then $G - uv$ is a smaller counterexample.

\section{Graphs with two disjoint cycles.}\label{sec:basecase}

In this section we will prove several lemmas that will serve as the base case $k=2$ for our various proofs. 
In notation, we will use convention~\eqref{pb1} with $k=2$.

\begin{lemma}\label{lemma:trianglefree}
Every triangle-free graph  $G$ with $h \geq \ell+ 4$ contains $2$ disjoint cycles.
\end{lemma}

\begin{proof}
Let $G$ be a minimal counterexample.  As $G$ is triangle-free, and $h \ge 4$, $|G|\ge8$. By Property~\ref{prop:minimal}, 
 $\delta (G) \geq 2$, and by Corollary~\ref{cor:K-K-Y}, $\delta(G)=2$. Say $d(x)=2$ and $N(x)=\{y,z\}$.
 By Property~\ref{prop:minimal}, $d(y), d(z) \in \{ 3,4 \}$.  Set $G'=G\diagup xy$. Since $G$ is triangle-free, $d_{G'}(v)=d_G(v)$ for all $v\in V(G) \setminus \{x, y \}$. As $x \in L$, this implies $|H_{2}(G')| \geq |L_{2}(G')| + 4$.  Since $G$ is minimal,
  $G'$ has a triangle, say $v_{xy} z w$. Then  $C:=xyzw$ is a $4$-cycle in $G$.  Let $W = V(G) \setminus C$.

As $x \in L$, $|C\cap H|-|L\cap H|\leq 2$.  So, since $h-\ell\geq 4$, $|H\cap W|-  |L\cap  W| \geq 2$.  Thus 
\begin{equation}\label{0128}
\sum_{u\in W}d(u)\ge3|W|+|H\cap W|-|L\cap  W|\ge3|W|+2.
\end{equation}

Each $v \in W$ has no two adjacent neighbors as $G$ is triangle free, and is not adjacent to
$x$  as $N(x)\subset C$. Thus if $\|v,C\|\ge2$ then $N(v)\cap C=\{y,z\}$. As $d(y)\le4$, there are at most two such $v$. So $\|W,C\|\le|W|+2$. Hence
by~\eqref{0128},
\[  2\|G[W]\|= \sum_{u\in W}d(u)-\|W,C\|
 \geq (3 |W| + 2) - (|W| + 2) = 2|W| .\]
Therefore $\|G[W] \| \geq |W|$, and so $G[W]$ contains a cycle (disjoint from $C$).
\end{proof}

The {\em 2-core} of a graph $G$ is the union of all  $G'\subseteq G$ with $\delta(G')\ge2$.  It can be obtained from $G$ by iterative deletion
of vertices of degree at most $1$.

\begin{lemma}\label{lemma:2core}
Suppose the $2$-core of $G$ contains at least $6$ vertices, and it is not isomorphic to $SK_{5}$.  If $h \geq \ell + 4$, then $G$ contains $2$ disjoint cycles.
\end{lemma}

\begin{proof}
Let $G$ be a minimal counterexample.  If there exists a vertex of degree at most $1$, then removing it yields a smaller counterexample.  So $G$ is its own $2$-core and $\delta(G)\ge2$. Also $|G|\ge6$ and by Corollary~\ref{cor:K-K-Y},  $L\ne \emptyset$. Thus $h\ge 5$, and $|G|\ge7$, since $G$ is not isomorphic to $SK_5$. 
 Pick $x\in L$. Let $N(x) = \{ y, z \}$. 
 
  Suppose $yz \notin E(G)$. Set $G'=G\diagup xy$. Then $|G'|=|G|-1 \ge 6$.  
Since $d(x) = 2$, all $y\in V(G')$ satisfy $d_{G'}(v) = d_{G}(v)$. So  $G'$ is its own $2$-core,  
 $|H_{2}(G')|-|L_2(G')| = h-\ell+1 \ge 5$, and $G'$ is not isomorphic to $SK_5$. 
    As $|G'|< |G|$, by the minimality of $G$, $G'$ has two disjoint cycles. But then so does $G$. 

Otherwise $yz \in E(G)$. Now $xyz$ is a triangle in $G$, so $G': = G -xyz$ is acyclic, and $\|G'\|<|G'|$.  Since $h \geq \ell + 4$ 
and $x\in L$, we have $|H \cap V(G')|-|L \cap V(G')|\ge3$. So $\sum_{v \in V(G')} d_{G}(v) \geq 3|G'| + 3$.  As $N(x)=\{y,z\}$, 
\[ \|V(G'),\{y,z\}\|=\|V(G'),V(G)\setminus V(G')\|\ge3|G'|+3-2(|G'|-1)\ge|G'|+5. \]
Thus $d(y),d(z)\ge6$. Let $G^*=G-x$. Then $|G^*|\ge6$, $d_{G^*}(y),d_{G^*}(z)\ge5$, and $d_{G^*}(v)=d_G(v)$ for all $v\in V(G^*)$. So $|H_2(G^*)|-|L_2(G^*)|\ge5$
and $G^*$ coincides with its $2$-core.
 As $|G^*|< |G|$, by the minimality of $G$, $G^*$ has two disjoint cycles. But then so does $G$.
\end{proof}

\begin{lemma}\label{lemma:1tr}
Every graph $G$ containing a triangle  $X=x_1x_2x_3$ has two disjoint cycles provided   
(a)   $|H \setminus X | -|L \setminus X| \geq  2$ and (b) $\|v,X\|\le2$ for all $v\in V(G)\setminus X$.
\end{lemma}

\begin{proof} Let $G$ be a minimal counterexample to the lemma. Then $G- X$ is acyclic. Let $Y=V(G)\setminus X$. 
By (a), there is $u \in H \setminus X$, and by (b)  $\|u,Y\| \geq 2$. This yields  $|G| \geq 6$. 
First, we show:
\begin{equation}\label{q1}
\mbox{
If $v\in Y$ and $\|v,Y\|\leq 1$, then $\|v,Y\|= 1$  and $\|v,X\|=2$.}
\end{equation}
Indeed, by (b), $\|v,X\|\leq 2$. So if $\|v,Y\|=0$ or
$\|v,X\|\leq 1$ and $\|v,Y\|\leq 1$, then $v\in L \setminus X$. Thus $G-v\subset G$  
satisfies (a) and (b). Then by the minimality of $G$, $G-v$ has two disjoint cycles, and hence so does $G$.

By (a), there are $z, z'\in H\setminus X$. 
If they are in the same component of $G-X$, then let $Q$ be the interior of the unique $z,z'$-path in $G-X$ and put $G'=G - X - Q - zz'$; 
otherwise put $G'=G-X$. Pick maximum paths $P=y_1\dots z\dots y_2$ and $P'=y'_1\dots z' \dots y'_2$ in $G'$.  
Perhaps $z=y_1$ or $z'=y_{1}'$, but $z,z'\in H$ implies $|P|, |P'| \geq 2$. 
 For $i \in \{1, 2\}$, if $y_{i} \neq z$ and $N(y_{i}) \cap Q \neq \emptyset$, then $G[P \cup Q]$ contains a cycle, a contradiction.  Then, 
\begin{equation}\label{q2}
\mbox{
either $d_{G-X}(y_{i}) = d_{G'}(y_{i})$ or $y_{i} = z$.}
\end{equation}
So, if $y_{i} \neq z$, then by~\eqref{q1}, $\|y_{i}, X\|  = 2$.  Otherwise $y_{i} = y_{1} = z$, and $\| z, X \| \geq d_{G}(z) - d_{G-X}(z) \geq d_{G}(z) - d_{G'}(z) - 1 \geq 2$.  So in any case, $\|y_{i}, X\| = 2$ and a similar argument shows $\|y_{i}', X\| = 2$.
Now $y_1$ and $y_2$ 
have a common neighbor, say $x_1$ in $X$, and $G[P+x_1]$ contains a cycle $C_1$. If $y'_1$ and $y'_2$ have a
common neighbor  $x_i\in X-x_1$, then $G[P'+x_i]$ contains a cycle disjoint from $C_1$. Otherwise, one of
$y'_1$ and $y'_2$ is adjacent to $x_2$ and the other to $x_3$. Then $G[P'+x_2+x_3]$ contains a cycle disjoint from $C_1$. 
\end{proof}

\section{Proof of Theorem~\ref{thm:2k+t}}\label{sec:2k+t}

Recall that we use convention~\eqref{pb1}.
Let $k$  be the smallest integer such that there exists a graph $G$ without $k$ disjoint cycles satisfying $|H | - |L | \geq 2k + t$ 
and $|G | \geq 3k$.  By Lemmas~\ref{lemma:trianglefree} and~\ref{lemma:2core}, $k \geq 3$.  Choose such $G$ to be minimal. 
 
 \begin{lemma}\label{lemma:3k} 
 $|G|\geq3k+1$.
 \end{lemma}

\begin{proof} 
Suppose that $|G| = 3k$. Create the graph $G' \supseteq G$ by adding to $G$ all edges necessary to make each vertex in $L$ a dominating vertex. 
Then $\delta(G') \geq 2k-1$, so by Corollary~\ref{cor:K-K-Y}, $G'$ contains $k$ disjoint cycles.  
As $|G'| = 3k$, these cycles are triangles, and at most $\ell$ of them contain edges from $E(G')\setminus E(G)$. 
Thus $t\geq k - \ell$  and so $h \geq \ell + 2k + t \geq 3k = |G|$. Hence $H=V(G)$ and by Theorem~\ref{thm:C-H}, $G$ contains $k$ disjoint cycles. 
\end{proof}

\begin{lemma}\label{lemma:Ltriangle}
Each $x\in L$  is in a triangle in $G$.
\end{lemma}

\begin{proof} Suppose $x$ is not in a triangle. By Property~\ref{prop:minimal}, $d(x) \geq 2$. Let $y \in N(x)$ and set $G' = G \diagup xy$.
Then $d_{G'}(v_{xy})=d(y)$ and
 $d_{G'} (z) = d_{G} (z)$ for all $z \in V(G')-v_{xy}$.
Since any triangle in $G'$ not containing $v_{xy}$ is also a triangle in $G$,  $t' := t_{G'} \leq t+1$. Thus
$H \subseteq H_{k}(G')$ and $L_{k}(G') +x \subseteq L$. So, 
\[
|H_{k}(G')|-|L_{k}(G')|\geq h-(\ell-1)\geq (\ell+2k+t)-\ell+1=1+2k+t\geq 2k+t'.
\]
By Lemma~\ref{lemma:3k}, $|G'|\geq3k$. As $G$ is minimal, $G'$ and $G$ have
$k$ disjoint cycles. 
\end{proof}

By Corollary~\ref{cor:K-K-Y}, $L\neq \emptyset$. Fix an $x \in L$.  Let $\T$ be a set of disjoint triangles in $G$ such that (a) $x \in V(\T)$, 
and (b) subject to (a),  $|\T|$ is maximum.  By Lemma~\ref{lemma:Ltriangle}, such a choice is possible. Let $T_{0} =T_0(\T)$ 
be the triangle in $\T$ containing $x$; say  $T_0=xyz$. 

Define an auxiliary digraph $\mathcal{D} = \mathcal{D}(\T)$ with $V(\mathcal{D}) = \T$ and $\overrightarrow{TU} \in E(\mathcal{D})$
 if and only if $T, U \in \T$ and  $\| v, U \| = 3$ for some $v \in T$. 
 If $v \in T$ and $\| v, U \| = 3$, we say the vertex $v$ \emph{witnesses} the edge $\overrightarrow{TU}$.
Let $\mathcal{R}=\mathcal R(\T) \subseteq \T$ be the set of triangles from which $T_{0}$ is reachable in  $\mathcal{D}(\T)$. Let $r = |\mathcal{R}|$. 
Since $T_{0} \in \mathcal{R}$,  $r \geq 1$.  Finally, define $B = B(\T) = \{ v \in V(G)\setminus V(\T) : \| v, T_{0} \| = 3 \}$.  
By the definitions of $\mathcal R$ and $B$, if $\| v, T_{0} \| = 3$ for a vertex $v$, then $v \in V(\mathcal{R})\cup B$.

\begin{lemma}\label{lemma:AdjacentR0}
If $|B| \leq 1$, then 
$\|v, T\| = 3$ for some vertex $v\notin V(\mathcal{R})\cup B$ and triangle $T \in \mathcal{R}$.
\end{lemma}

\begin{proof} Suppose $|B| \leq 1$ and $\| v, T \| \leq 2$ for every $v \notin V(\mathcal{R}) \cup B$ and $T \in \mathcal{R}$.

\medskip{}
\noindent \textbf{Case 1:} \emph{$r \leq k-2$}. Let $G' = G - V(\mathcal{R})$ and observe $t_{G'} \leq t - r$. 
We will find $k' := k - r$ disjoint cycles in $G'$. For each $v \notin V(\mathcal{R}) \cup B$, $\| v, V(\mathcal{R}) \| \leq 2r$, so $d_{G'}(v) \geq d_{G}(v) - 2r$. Thus $H \setminus ( V(\mathcal{R}) \cup B) \subseteq H_{k'}(G')$ and $L_{k'}(G')\subseteq (L\setminus V(\mathcal{R})) \cup B$.  As $x\in L\cap V(\mathcal{R})$ and $|B| \leq 1$,
\begin{align*}
|H_{k'}(G')| & \geq h - (3r - 1) - |B|\geq h-3r\quad\mbox{and}\quad|L_{k'}(G')| \leq (\ell - 1) + |B|\leq \ell.
\end{align*}
Combining these inequalities  yields 
\[
|H_{k'}(G')|- |L_{k'}(G')| \geq (h - 3r)-\ell \geq  2 (k - r) + (t - r) \geq  2k' + t_{G'}.
\]
As $k'\geq2$ and $|G'|=|G|-3r \geq3k'$, $G'$ contains $k'$ disjoint cycles by the minimality of $G$, and thus $G$ has $k$ disjoint cycles.

\medskip{}
\noindent \textbf{Case 2:} $r=k-1$. Let $\mathcal{R}^{-}=\mathcal{R} - T_{0}$ and consider $G' = G - V(\mathcal{R}^{-})$.  For each $v \notin V(\mathcal{R}) \cup B$, since $\| v, V(\mathcal{R}^{-}) \| \leq 2 (r - 1)$, $d_{G'}(v) \geq d_{G}(v) - 2k + 4$.
This implies that $H \setminus (V(\mathcal{R}) \cup B) \subseteq H_{2}(G') - T_{0}$ and,
since each vertex in $B$ is adjacent to three vertices in $T_{0} \subseteq G'$,
$L_{2}(G') - T_{0} \subseteq L\setminus V(\mathcal{R})$. Therefore, since $x \in L \cap V(\mathcal{R})$ and $|B| \leq 1$,
\begin{align*}
|H_{2}(G') - T_{0}| & \geq h - (3r - 1) - |B| \geq h - 3k + 3\quad\mbox{and}\quad |L_{2}(G') - T_{0}| \leq \ell - 1.
\end{align*}
Since  $t = k - 1$, these inequalities give 
\[
|H_{2}(G') - T_{0}| -|L_{2}(G') - T_{0}|\geq (h - 3k + 3)- \ell+1\geq (\ell + 2k + (k - 1)) - 3k + 3-\ell+1 = 3.
\]
If $\| u, T_{0} \| = 3$, then $u$ is the unique vertex in $B$; in this case let $e$ be an edge from $u$ to $T_{0}$  and let $G''=G' - e$.  
Otherwise, let $G'' = G'$. Since in both cases, $d_{G''}(v) = d_{G'}(v)$ for $v \in V(G') - T_{0} - u$,  $|H_{2}(G'') - T_{0}| - |L_{2}(G'') - T_{0}| \geq 2$. 
By Lemma~\ref{lemma:1tr}, $G''$ contains two disjoint cycles, and so $G$ contains $k$ disjoint cycles, a contradiction.
\end{proof}

\begin{lemma}\label{lemma:maxB}
If $v \notin V(\mathcal{R}) \cup B$ and $\| v, T \| = 3$ for some $T \in \mathcal{R}$,
then there are a vertex $v' \in V(\T) $ and a set $\T'$ of disjoint triangles such that
$xyz \in \T'$, $| \T' | = | \T |$, $B(\T')=B+v'$, and $V(\T') = V(\T) + v - v'$. 
\end{lemma}

\begin{proof}Let $T=T_{j},T_{j-1},\ldots,T_{0}$ be a $T \to T_{0}$ path in $\mathcal{D}(\T)$
and, for each $i\in[j]$, let $v_{i}$ witness the edge $\overrightarrow{T_{i} T}_{i-1}$.
Define $T_{j}' = T_{j} - v_{j} + v$ and $T_{i}' = T_{i} - v_{i} + v_{i+1}$ for all $i \in [j-1]$.
Then, $\T' = \T - \{T_{1}, \ldots, T_{j} \} + \{ T_{1}', \ldots, T_{j}' \}$ is a set of $|\T|$
disjoint triangles in $G$, $v':=v_{1}\notin V(\T')\cup B$, and $\|v',T_{0}\|=3$.
Thus $B+v'=B(\T')$. \end{proof}

Now choose $\T$ subject to (a) and (b) so that $B$ is maximum.

\begin{lemma}\label{lemma:K5} $|B| = 2$.  Moreover, $\|v, T_{0} \cup B\| \leq 2$
for all $v \notin V(\T) \cup B$. \end{lemma}

\begin{proof}As $B$ is maximum, Lemmas~\ref{lemma:AdjacentR0} and~\ref{lemma:maxB} imply $|B| \geq 2$. Fix a vertex $u_{1} \in B$ and let $T_{0}'$ be the triangle $xyu_{1}$.  Observe $\T' = \T - T_{0} + T'_{0}$ is a set of $|\T|$ disjoint triangles in $G$.  Let $\mathcal{R}'=\mathcal{R(\mathcal T')}$, $r'=|\mathcal{R}'|$, $B'=B'(\T')$ and note $z \in B'$.  If $|B'| \geq 2$, let $\T'' = \T'$. Otherwise by Lemma~\ref{lemma:AdjacentR0}, there are $v \notin V(\mathcal{R}') \cup B'$ and $T \in \mathcal{R}'$ with $\|v, T \| = 3$.  By Lemma~\ref{lemma:maxB}, there are $z' \in V(\T')$ and a set $\T''$ of triangles  satisfying $T_{0}' \in \T''$, $|\T''| = | \T' |$, and $B(\T'') = \{z, z' \}$.  

If $|B|\ge 3$ then pick $u_{2} \in B - u_{1} - v$. As $V(\T'') \setminus V(\T ) \subseteq \{u_{1}, v \}$,  $u_{2} \notin V(\T'')$.  Thus $\T'' - T'_{0} + xu_{1}z' + yu_{2}z$ is a set of $|\T|+1$ disjoint triangles containing $x$, contradicting (b). So $|B| = 2$. 

Lastly, if $v \notin V(\T) \cup B$ and $\| v, T_{0} \cup B \| \geq 3$, then $v$ has neighbors $w\in T_{0}$ and $u\in B$.
  Thus $vuw$ and $(T_0-w)  \cup (B-u)$ are disjoint triangles in $(T_{0} \cup B) + v$, contradicting (b).
\end{proof}

Let $T^* := G[T_{0} \cup B]$.  Define a second auxiliary digraph $\mathcal{D}^* (\T)$ to have vertex set $\T - T_{0} + T^*$ 
and $\overrightarrow{TU} \in E(\mathcal{D}^* (\T))$ if and only if $\| v, U \| \geq 3$ for some $v \in T$.  
Again, we say the vertex $v$ \emph{witnesses} the edge $\overrightarrow{TU}$.   Define the set of graphs $\mathcal{R}^*$ to be $T^*$ 
together with the set of triangles from which 
$T^*$ is reachable in $\mathcal{D}^*(\T)$. Let $r^* = |\mathcal{R}^*|$.

\begin{lemma}\label{lemma:Rbar}
If $v \in V(G)\setminus V(\mathcal{R}^*)$, then $ \| v, T \| \leq 2$ for each $T \in \mathcal{R}^*$.
\end{lemma}

\begin{proof} Suppose $v \in V(G)\setminus V(\mathcal{R}^*)$, $T \in \mathcal{R}^*$, and $ \| v, T \| \geq 3$.
Let $T = T_{j}, T_{j-1}, \ldots, T_{1}, T^*$ be a $T \to T^*$ path in $\mathcal{D}^*(\T)$.  By Lemma~\ref{lemma:K5}, $v$ is adjacent to at most $2$ vertices in $T^*$, so $j \geq 1$.

Let $v_{1}$ witness the edge $\overrightarrow{T_{1} T}^*$ and, for $i \in \{2, \ldots, j \}$, let $v_{i}$ witness the edge $\overrightarrow{T_{i} T}_{i-1}$.  As in the proof of Lemma~\ref{lemma:maxB}, define $T_{j}' = T_{j} - v_{j} + v$ and $T_{i}' = T_{i} - v_{i} + v_{i+1}$ for all $i \in [j-1]$.  If $\|v_{1}, T_{0} \| = 3$, then $\T' = \T - \{T_{1}, \ldots, T_{j} \} + \{ T_{1}', \ldots, T_{j}' \}$ is a set of $|\T|$ triangles in $G$, but $B+v_{1}=B(\T')$, contradicting the maximality of $B$.  Otherwise, there exist a vertex 
$w \in N(v_{1}) \cap T_{0}$, a vertex $u \in N(v_{1}) \cap B$, and a triangle $T_{0}' =(T_{0} - w) \cup (B - u)$.  
Then $\T' = \T - \{T_{0}, T_{1}, \ldots, T_{j} \} + \{ v_{1}wu, T_{0} ', \ldots, T_{j}' \}$ is a set of $|\T| + 1$ disjoint triangles in $G$, contradicting the maximality of $\T$.
\end{proof}

\begin{proof}[Proof of Theorem~\ref{thm:2k+t}]

Let $G'=G-V( \mathcal{R}^*)$. Set $k'=k - r^*$ and $t' = t - r^*$. Then $k' \geq 1$.
By Lemma~\ref{lemma:K5}, $B$ has the form $\{w_{1},w_{2}\}$, and by Lemma~\ref{lemma:Rbar}, every
$v \in V(G')$ satisfies 
\begin{equation}
d_{G'}(v)\geq d_{G}(v)-2r^*.\label{deg-eq}
\end{equation}
Thus $H \setminus V(\mathcal{R}^*)\subseteq H_{k'}(G')$ and $L_{k'}(G')\subseteq(L\setminus V(\mathcal{R}^*))$.
As $x\in L \cap T_{0}$, this implies
\begin{align*}
|H\cap H_{k'}(G')| & \geq h- (3r^* - 1)-|B| \geq h - 3 r^* - 1\quad\mbox{and}\quad |L_{k'}(G')|\leq  |L\cap V(G')|\leq\ell - 1.
\end{align*}
Combining these inequalities yields 
\begin{equation}
|H_{k'}(G')| - |L_{k'}(G')| \geq (h - 3 r^* - 1) - (\ell - 1) \geq 2 (k - r^*) + (t - r^*) = 2k'+t'.\label{hyp-eq}
\end{equation}
and 
\begin{equation}
|H\cap H_{k'}(G')| - |L\cap V(G')| \geq (h - 3 r^* - 1) - (\ell - 1) \geq 2 (k - r^*) + (t - r^*) = 2k'+t'.\label{hyp-eq'}
\end{equation}

\textbf{Case 1:} $|G'|=3k'-1$. As $H_{k}'(G')\ne\emptyset$, $\Delta(G')\geq2k'$ and $2k'+1\leq|G'|=3k'-1$. So $k'\geq2$. Let $G^{+}=G'\vee K_{1}$, where $V(K_{1})=\{u\}$.
Then $|G^{+}|=3k'$ and $t_{G^{+}}\leq t'+1$. So 
\[
|H_{k'}(G^{+})| - |L_{k'}(G^{+})| \geq |H_{k'}(G')+u| - |L_{k'}(G^{'})| \geq 2k' + t' + 1 \geq 2k' + t_{G^{+}}.
\]
As $|G|$ is minimal, $G^{+}$ has a set $\mathcal{S'}$ of $k'$ disjoint triangles. Since $|G^{+}|=3k'$, we may assume
 $T'=uu_{1}u_{1}'\in\mathcal{S}'$. Let $T^{''}=xyw_{1}$
and $\mathcal{S}=(\mathcal{S}'-T')\cup(\mathcal{R}-T_0+T^{''})$. Thus $t=k-1$, $h \geq 2k + t + \ell=3k$
and $\ell = 1$. So $H=V(G)-x$. Let  $u_{2}u_{2}':=zw_{2}$,
$U=\{u_{1},u_{1}',u_{2},u_{2}'\}$, and note that $u_{1}u_{1}',u_{2}u_{2}'\in E(G-V(\mathcal{S})$). 

Since $U\subseteq H$, and $G[U]$ is acyclic, $\|U,V(G)\setminus U\|\ge8k-6>8(k-1)$. Thus
 $\|U,T\|\geq9$ for some $T=q_1q_2q_3\in\mathcal{S}$. 
 Say $\|q_1,U\| \le \|q_2,U\| \le \|q_3,U\|$.
 Then  $q_2u_{i}u_{i}'$ is a triangle for some $i\in [2]$. Now $\|\{q_1,q_3\},\{u_{3-i},u_{3-i}'\}\|\ge2$, so $\{q_1,q_3,u_{3-i},u_{3-i}'\}$ contains a cycle. Thus $G$ has $k$ disjoint cycles, a contradiction.

\textbf{Case 2:} $|G'|\geq3k'$. If $k'\geq2$ then \eqref{deg-eq}, \eqref{hyp-eq},
and the minimality of $G$ imply $G'$ contains $k'$ cycles and so $G$ contains $k$
cycles. So assume $k'=1$ and $G'$ is acyclic. 

By \eqref{hyp-eq'}, $|H\cap H_{k'}(G')| - |L\cap V(G')| \geq 2$. Thus, there is a component 
 $G_{0}$ of $G'$ with 
 \begin{equation}\label{Neq}
|H\cap H_{k'}(G_0)| - |L\cap V(G_0)| \geq 1.
\end{equation}
 By~\eqref{deg-eq}, $|G_0|\geq 3$.
 Let  $W_0=V(G_0)$ and $G'_0=G[T^*\cup W_0]$. By Lemma~\ref{lemma:K5} and the fact that $G_0$ has no isolated vertices,
 \[
 d_{G'_0}(v)\geq\left\{\begin{array}{ll}4,& \mbox{ if } v\in H\cap W_0;\\
 1,& \mbox{ if } v\in L\cap W_0;\\
3 ,& \mbox{ if } v\in  W_0-L-H.
 \end{array}\right.\]
 By this and~\eqref{Neq},
 \begin{align*}
 \|W_0,T^*\| = \sum_{v\in W_0}d_{G'_0}(v)-2\|G_0\| &\geq 2.5|W_0|+1.5(|H\cap W_0|-|L\cap W_0|)-2(|W_0|-1) \\
&\geq 0.5|W_0|+1.5+2\geq 5.
\end{align*}
It follows that there are $w\in T_{0}$ and $u\in B$ such that $\|\{w,u\},W_0\|\geq 2$. Then
$G[W_0+w+u]$ contains a cycle, and $(T_0-w)  \cup (B-u)$ induces a triangle. 
This gives $k$ disjoint cycles. \end{proof}

\section{Proof of Corollary~\ref{cor:2k+t+1}}\label{sec:2k+t+1}

Suppose an integer  $k\ge2$ and a graph $G$ satisfy $h - \ell \geq 2k + t + 1$, and $G$ has no $k$ disjoint cycles.  
By Lemmas~\ref{lemma:trianglefree} and ~\ref{lemma:2core}, $k \geq 3$. Let $|G|=3k'+r$, where $k' = \left\lfloor k / 3 \right\rfloor$ and $0 \le r \le2$. 
 By Theorem~\ref{thm:2k+t}, $3k-1\ge|G|\ge h\geq 2k+1\ge7$, so $ k-1\ge k'\ge2$. 
  Pick $R\subset V(G)$ so that  $G':=G-R$ has $t$ disjoint triangles.  Let $r=|R|$. Then $t_{G'}=t$, and  $d_{G'}(v)\ge d_G(v)-2$ for each $v \in V(G')$. Thus
\[ |H_{k'}(G' )| - |L_{k'}(G ')| \ge |H \setminus R| - \ell \ge 2k+t+1-r\ge2k'+t_{G'}+1.\]
By Theorem~\ref{thm:2k+t},  $G'$  has $k'$ disjoint triangles, so $t_{G'}= k'$ and   $|H_{k'}(G')|\ge3k'+1>|G'|$, a contradiction.

\section{Proof of Theorem~\ref{thm:planar}}\label{sec:planar}

The proof will be by contradiction.  Consider the smallest $k$ such that there exists a counterexample 
$G$, and choose such  $G$ to be minimal. 
 If $k = 2$, then $h \geq 4$, so $G=K_5$ or $|G| \geq 6$.
 As $G$ is planar, $G$ contains neither $K_5$ nor $SK_5$. Thus by Lemma~\ref{lemma:2core}, $G$ has two disjoint cycles. Hence $k \geq 3$.

We first show that $L \neq \emptyset$.  Since $G$ is planar, $\| G \| \leq 3 |G| - 6$ and the average degree is less than $6$.  
If $k \geq 4$, then $L \neq \emptyset$ follows immediately.  If $k = 3$ and $\delta(G) = 5$, then since $h \geq 2k=6$,  
$\| G \| \geq \frac{1}{2} ( 36 + 5 ( |G| - 6))$.  This implies $|G| \geq 18=6k$, and by Corollary~\ref{cor:K-K-Y}, $L\neq \emptyset$. 

Let $x \in L$. We claim that
\begin{equation}\label{0129}
 \mbox{for every $y \in N(x)$, the edge $xy$ is contained in a triangle. }
\end{equation}
Indeed, if $xy$ is not  in a triangle, then consider the graph $G^*=G \diagup xy$.
The degree of every vertex other than $x$ or $y$ remains unchanged and the degree 
of $v_{xy}$ is at least the degree of $y$.  Therefore, $|H_{k}(G^*)| 
\geq |L_{k}(G^*)| + 2k$ and by the minimality of $G$, $G^*$ contains  $k$ disjoint cycles. 
 Expanding the edge $xy$ yields $k$-disjoint cycles in $G$. This proves~\eqref{0129}.

Fix a plane drawing of $G$.  Every triangle $T$  separates the plane into  the exterior region $R_{1}(T)$ and interior region $R_{2}(T)$.
Among all triangles containing $x$, choose $T'$  so that  $R_2(T')$ contains the fewest vertices.  Let $T' = xyz$, $R_1=R_1(T')$ and $R_2=R_2(T')$.
By~\eqref{0129}, $R_2$ contains no neighbors of $x$.

Suppose $G$ has two vertices $v_{1}$ and $v_{2}$ adjacent to all three vertices of $T'$. By the choice of $T'$ and $R_2$, both $v_{1}$ and $v_{2}$ are in $R_1$.
The planar drawing induced by $T' \cup \{v_{1}, v_{2} \}$ contains no edges in the interior of $R_{2}$.  Adding a vertex $v$ in $R_{2}$ adjacent to all three vertices of $T'$ gives a planar embedding of $K_{3,3}$, a contradiction.  So  $G$ has at most one vertex $v_{1}$ adjacent to all $3$ vertices of $T'$.

Let $G' = G - T'$, $k' = k - 1$.  Then for each $u \in V(G) -v_{1} $, $d_{G'} (u) \geq d_{G} (u) - 2$ and $d_{G'}(v_{1}) = d_{G}(v_{1}) - 3$. 
It follows that $|H\cap \{v_{1}\}|+|L_{k'}(G') \cap \{v_{1}\}| \leq 1$. Hence
\[ |H_{k'}(G')|-|L_{k'}(G')| \geq (h - 2 - |H \cap \{v_{1}\}|)-(\ell - 1+ |L_{k'}(G') \cap \{v_{1}\}|)\geq 2k - 2= 2k'.\]
By the minimality of $G$, $G'$ contains $k-1$ disjoint cycles, and so $G$ contains $k$ disjoint cycles.

\section{Proof of Theorem~\ref{thm:1tri}}\label{sec:1tri}

Following Dirac and Erd\H{o}s \cite{D-E}, let $V_{\geq s}(G)$ (respectively, $V_{\leq s}(G)$) denote the
set of vertices of $G$ of degree at least $s$ (respectively, at most $s$). In these terms, $H=H_k(G)=V_{\geq 2k}(G)$ and $L=L_k(G)=V_{\leq 2k-2}(G)$.
The following lemma  may be of interest on its own. 

\begin{lemma}\label{lemma:2k-1}
Let $G$ be a triangle-free graph with $V(G) \neq \emptyset$.  
 If 
\begin{equation}\label{2k+1} 
 |V_{\geq 2k+1} (G)| - |V_{\leq 2k-1} (G)| \geq 2k - 2,
\end{equation} 
  then $G$ has $k$ disjoint cycles.
\end{lemma}

\begin{proof}
 Suppose the lemma does not hold and consider the smallest $k$ such that there exists a counterexample. 
 Among all such counterexamples, choose the graph $G$ to be minimal. 
First consider $k = 1$.  Since $|V_{\geq 3} (G)| \geq |V_{\leq 1} (G)|$, $G$ contains a component with average degree at least $2$.  
Therefore, $G$ contains a cycle and the claim holds. Now, let $k\geq 2$.

By~\eqref{2k+1}, the sum of the degrees of  vertices in $V_{\geq 2k} (G)$ is greater than the sum of degrees of vertices in $V_{\leq 2k-1} (G)$. 
Thus there are
 vertices $u,v \in V_{\geq 2k} (G)$ such that $uv \in E(G)$.  
 Since $G$ is triangle-free, $N(v) \cap N(u) = \emptyset$ and so $|G| \geq 4k$.  
 Since $G$ has no $k$ disjoint cycles, by Theorem~\ref{thm:C-H},  $G$ has a vertex $x \in V_{\leq 2k-1} (G)$.

As in Property~\ref{prop:minimal}, if  $d(x) \leq 1$, then $G - x$ is a smaller counterexample, so $d(x) \geq 2$. Let $y \in N(x)$.  Since $G$ is triangle-free,
contracting the edge $xy$ does not change the degree of any vertex distinct from $x,y$. 
By the minimality of $G$, $G \diagup xy$ contains  either $k$ disjoint cycles or  a triangle. 
 If $G \diagup xy$ contains $k$ disjoint cycles, then $G$ does as well.  Otherwise, let $v_{xy} zw$ be a triangle in $G \diagup xy$. 
Then by symmetry we may assume  $xyzw$ is a 4-cycle in $G$.  Every vertex in $G - \{ w,x,y,z\}$ is adjacent to at most $2$ vertices in $\{w,x,y,z\}$.

Let $k' = k - 1$ and $G' = G - \{w,x,y,z\}$.  Then, for each $v \in V(G')$, $d_{G'} (v) \geq d_{G}(v) - 2$, so $|V_{\geq 2k' + 1} (G')| \geq |V_{\geq 2k + 1} (G)| - 3$ and 
$|V_{\leq 2k' -1}(G')| \leq |V_{\leq 2k -1} (G)| - 1$.  Therefore,
\begin{equation*}\label{eq:2k-1}
|V_{\geq 2k'+1} (G')| -|V_{\leq 2k' - 1} (G')|\geq |V_{\geq 2k + 1} (G)| - 3- ( |V_{\leq 2k -1} (G)| - 1)\geq  2k -2 - 2 = 2k' - 2.
\end{equation*}
By the minimality of $G$, $G'$ contains $k'$ disjoint cycles. Hence $G$ contains $k$ disjoint cycles.
\end{proof}

Suppose that Theorem~\ref{thm:1tri} is false and let $k$ be the smallest integer such that there exists a a counterexample. 
 Among all counterexamples, choose $G$  to be minimal.
\begin{lemma}\label{lemma:4k-1}
$|G| \geq 4k-1$ and $L \neq \emptyset$.
\end{lemma}

\begin{proof}
Suppose $|G| \leq 4k - 2$. For all $u \in H$, $|N(u) \cap H| \geq 2$ and if also $w \in H$ then $|N(w) \cap N(u)| \geq 2$.
  It suffices to show that $G$ has two disjoint triangles.  As $h \geq 2k \geq 6$, if $G[H]$ is a complete graph, then we are done, so assume there are $x,y \in H$ with $xy \notin E(G)$. 
  
  Choose $w \in N(x) \cap H$, $z\in N(y)\cap H - w$, and $v\in N(w) \cap N(x) -z$.  If $N(y) \cap N(z) \neq \{v,w\}$, then there are two triangles in $G$;  else put $Q=\{v,w,y,z\}$ and $P= N(x) \setminus Q$. 
Now $|P|\geq2k-3\geq k$. If there is $u\in P$ with $d(u)\geq2k-1$, then there
is $t\in N(x)\cap N(u)$. Thus $txu$ is a triangle, and $Q -t$
contains another  triangle. So $V(P) \subseteq L$ and $|L | \geq k$.  Therefore, $|G|\geq h + \ell \geq 2\ell+2k \geq 4k$. 
\end{proof}

\begin{lemma}\label{lemma:LNoTriangle}
If $x \in L$, then $x$ is not contained in a triangle.
\end{lemma}

\begin{proof}
Let $x \in L$ and suppose $T_{0}$ is a triangle in $G$ containing $x$.  Let $B = B(T_{0}) = \{ v \in V(G) : \|v, T_{0} \| = 3 \}$ and fix $T_{0} = xyz$ so that $|B|$ is minimized.  Let $k' = k-1$, $G' = G - T_{0}$.  For each $v \in V(G') - B$, $d_{G'}(v) \geq d_{G}(v) - 2$, so $|H_{k'}(G')| \geq |H \setminus (B \cup T_{0})|$ and $|L_{k'}(G')| \leq |L \setminus (B \cup T_{0})|\leq \ell-1$.

If $|B| \leq 1$, then $|H_{k'}(G')|-|L_{k'}(G')| \geq (h - 3)-(\ell-1) \geq  2k - 2 = 2k'$. 
 Since $G'$ is triangle-free, by Theorem~\ref{thm:2k+t}, $G'$ contains $k-1$ disjoint cycles. Then $G$ contains $k$ disjoint cycles. 
 Similarly, if $|B| = 2$ and $B \cup T_{0}$ contains at most $3$ vertices in $H$, then $G'$ contains $k$ disjoint cycles.  
So we may assume that $|B| \geq 2$ and $B \cup T_{0}$ contains at least $4$ vertices in $H$.  We complete the proof in $3$ cases.

\medskip
\noindent \textbf{Case 1:} \emph{$B$ is an independent set.}  Let $u_{1}, u_{2} \in B$ and $T_{1} = xyu_{1}$.  If $v \notin B \cup V(T_{0})$ 
and $\| v, T_{1} \| = 3$, then $xu_{1}v$ and $yzu_{2}$ are two disjoint triangles in $G$. 
 Let $k' = k- 1$, $G'' = G - T_{1}$.  For each $v \in V(G'') - z$, $d_{G''}(v) \geq d_{G}(v) - 2$ and $d_{G''}(z) = d_{G}(z) - 3$.  
So possibly $z \in H \setminus H_{k'}(G'')$ or $z \in L_{k'}(G'') \setminus L$, but not both, i.e., $| \{z\} \cap H|+ | \{z\}  \cap L''|\leq 1$.  Therefore,
\begin{align}
|H_{k'}(G'')| -|L_{k'}(G'')| &\geq (h - 2 - | \{z\} \cap H|)-(\ell-1+  | \{z\}  \cap L_{k'}(G'')|\nonumber \\
&\geq (h-\ell) -1 - (| z \cap H |+ | \{z\}  \cap L_{k'}(G'')|) \label{eq:LNoTriangle1} \\
&\geq 2k-2=2k'. \nonumber  
\end{align}
By Theorem~\ref{thm:2k+t}, $G''$ contains $k-1$ disjoint cycles. Then $G$ contains $k$ disjoint cycles.

\medskip
\noindent \textbf{Case 2:} \emph{$|B| \geq 3$.}  Let $u_{1}, u_{2}, u_{3} \in B$ and, by Case~1 assume $u_{1}u_{2} \in E(G)$.   Then $xu_{1}u_{2}$ and $yzu_{3}$ are two triangles in $G$, a contradiction.

\medskip
\noindent \textbf{Case 3:} \emph{$|B| =2$.}  Let $u_{1}, u_{2} \in B$ and, by Case~1, assume $u_{1}u_{2} \in E(G)$. 
 In particular $B \cup T_{0} \cong K_{5}$ and every vertex in $B \cup T_{0}$ apart from $x$ is in $H$.  If $v \notin B \cup T_{0}$ is adjacent to $2$ vertices in $B \cup T_{0}$, then $G$ contains $2$ disjoint triangles, a contradiction.  Let $k' = k - 1$ and $G' = G - (B \cup T_{0})$.  For each $v \in V(G')$, $d_{G'} (v) \geq d_{G} (v) - 1$.  In particular, $|V_{2k'+1}(G')| \geq h - 4$ and $|V_{2k'-1}(G')| \leq \ell - 1$.  Therefore,
\begin{equation}\label{eq:LNoTriangle3}
|V_{2k'+1}(G')| -|V_{2k'-1} (G')|\geq (h - 4)-(\ell-1) \geq  2k - 3 = 2k' - 1.
\end{equation}
The graph $G'$ is triangle-free, so by Lemma~\ref{lemma:2k-1}, $G'$ contains $k-1$ disjoint cycles. Then $G$ contains $k$ disjoint cycles.
\end{proof}

\begin{lemma}\label{lemma:C4}
If $x,z \in L$, then $|N_{G}(x) \cap N_{G}(y)| \leq 1$.
\end{lemma}

\begin{proof}
Suppose $w,y \in N_{G}(x) \cap N_{G}(z)$. Then $X = wxyz$  is a copy of $C_{4}$ in $G$.  If $v \notin X$ is adjacent to at least $3$ vertices in $X$, then either $x$ or $z$ is contained in a triangle, contradicting Lemma~\ref{lemma:LNoTriangle}.  Let $G' = G - X$.  For each $v \in V(G')$, $d_{G'}(v) \geq d_{G}(v) - 2$.  Therefore,
\begin{equation}\label{eq:C4}
|H_{k'}(G')|- |L_{k'}(G')|  \geq (h - 2)-(\ell-2) \geq  2k = 2k' + 2.
\end{equation}
Since $G'$ contains at most $1$ triangle, by Theorem~\ref{thm:2k+t}, $G'$ contains $k-1$ disjoint cycles. Then $G$ contains $k$ disjoint cycles.
\end{proof}

Let $L = \{x_{1}, \ldots, x_{\ell} \}$ and, for each $i$, let $y_{i} \in N_{G}(x_{i})$.  Starting with the graph $G = G_{0}$, we construct a sequence of graphs by defining $G_{i} = G_{i-1} \diagup x_{i}y_{i}$.  For simplicity, if we contract the edge $x_{i}y_{i}$, we label the contracted vertex in $G_{i}$ as $y_{i}$.   We terminate this process if $G_{i}$ contains $k$ cycles or when $i = \min \{ \ell, k - 1\}$.  Suppose, after terminating the process, we have defined graphs $G_{0}, \ldots, G_{r}$ for some non-negative integer $r$.

\begin{lemma}\label{lemma:induction}
For the graphs $G_{0}, \ldots, G_{r}$ and $i \in \{0, \ldots, r\}$, all of the following hold:
\begin{enumerate}
\item $|G_{i}| = |G_{0}| - i \geq 3k$;
\item if $i < r$, then $G_{i}$ contains $i+1$ disjoint triangles;
\item $L_{i}$ is an independent set;
\item if $x \in L_{k}(G_{i})$, then $N_{G_{i}}(x)$ is an independent set;
\item if $x,x' \in L_k(G_{i})$, then $|N_{G_{i}}(x) \cap N_{G_{i}}(x')| \leq 1$;
\item  $L_{k}(G_{i}) = L_{0} - \{ x_{1}, \ldots, x_{i} \}$ and $H_{k}(G_{i}) \supseteq H_{k}(G_{0})$;
\item if $i \geq 1$ and $G_{i}$ contains $k$ disjoint cycles, then $G_{i-1}$ does as well.
\end{enumerate}
\end{lemma}

\begin{proof}  For all $i$,  \emph{(1)} holds by Lemma~\ref{lemma:4k-1} and \emph{(7)} holds since a contraction cannot increase the number of disjoint cycles. 

The proof of \emph{(2)--(6)} will be by induction on $i$.    By assumption, $G_{0}$ contains at most $1$ triangle.  If $G$ is triangle-free, then by Theorem~\ref{thm:2k+t}, $G_{0}$ contains $k$ disjoint cycles, so \emph{(2)} holds for $i=0$.  Since $G_{0}$ is a minimum counterexample, \emph{(3)} holds for $i=0$ by Property~\ref{prop:minimal}.  Further, \emph{(4)} and \emph{(5)} hold for $i=0$ by Lemma~\ref{lemma:LNoTriangle} and Lemma~\ref{lemma:C4}, respectively.
And \emph{(6)} is trivial for $i=0$.

Suppose that $r \geq 1$ and consider $i \in \{ 1, \ldots, r\}$.  Assume that \emph{(2) - (6)} hold for all $j < i$. Recall that $G_{i} = G_{i-1} \diagup x_{i} y_{i}$.  
By \emph{(4)} for $i-1$, $d_{G_{i}} (y_{i}) \geq d_{G_{i-1}}(y_{i})$ and no other vertex $v$ is  adjacent to both $x_{i}$ and $y_{i}$, so $d_{G_{i}}(v) = d_{G_{i-1}}(v)$.  Thus, \emph{(6)} holds.

To see that \emph{(3)} holds, observe if $uv \notin E(G_{i-1})$, then $uv \in E(G_{i})$ only if $u, v \in N_{G_{i-1}}(x_{i})$.  Since $L_{i-1}$ is an independent set and $L_{k}(G_{i}) \supseteq L_{k}(G_{i-1})$ by \emph{(6)}, $L_{k}(G_{i})$ is also an independent set.  

If $x \in L_{k}(G_{i})$, then by \emph{(6)}, $x \in L_{k}(G_{i-1})$ also and $x \neq x_{i}$.  Let $y, y' \in N_{G_{i}}(x)$ and note that since \emph{(4)} holds for $G_{i-1}$, $yy' \notin E(G_{i-1})$.  Edges are only added to $G_{i}$ between pairs of vertices in $N_{G_{i-1}}(x_{i})$.  Since \emph{(5)} holds for ${i-1}$, $|N_{G_{i-1}} (x_{i}) \cap N_{G_{i-1}}(x)| \leq 1$, so $y$ and $y'$ cannot both be in $N_{G_{i-1}} (x_{i}) \cap N_{G_{i-1}}(x)$.  Thus, $yy' \notin E(G)$ and \emph{(4)} holds for ${i}$.

If $x, x' \in L_{k}(G_{i})$, then by \emph{(6)}, $x,x' \in L_{k}(G_{i-1})$ and $|N_{G_{i-1}} (x) \cap N_{G_{i-1}}(x') | \leq 1$.  Since $L_{k}(G_{i-1})$ is an independent set, $N_{G_{i}}(x) = N_{G_{i-1}}(x)$ and $N_{G_{i}}(x') = N_{G_{i-1}}(x')$, so $|N_{G_{i}} (x) \cap N_{G_{i}}(x') | \leq 1$ and \emph{(5)} holds.

Finally, by \emph{(2)}, $G_{i-1}$ contains exactly $i$ disjoint triangles.  Contracting an edge introduces increases the number of disjoint triangles
by at most $1$, so $G_{i}$ contains 
at most $i+1$ disjoint triangles.  By \emph{(6)},
\begin{equation}\label{eq:induction1}
|H_{k}(G_{i})|-|L_{k}(G_{i})| \geq h -(\ell-i)\geq  2k + i.
\end{equation}
Since $|G_{i}| \geq 3k$, if $G_{i}$ contains $i$ disjoint triangles, by Theorem~\ref{thm:2k+t}, $G_{i}$ contains $k$ disjoint cycles and $i = r$. 
 Therefore, if $i < r$ then $G$ contains exactly $i+1$ disjoint triangles and \emph{(2)} holds.
\end{proof}

We are now ready to complete the proof.  If $r < \min \{ \ell, k - 1 \}$, then we stopped the process because $G_{r}$ contains $k$ disjoint cycles.  
 If $r = k - 1 = \min \{ \ell, k - 1 \}$, then $G_{k-2}$ contains $k-1$  disjoint triangles and $G_{k-1}$ contains at least this many disjoint
 triangles.  If $G_{k-1}$ contains only $k-1$  disjoint triangles, then by Lemma~\ref{lemma:induction} \emph{(6)},
\begin{equation}\label{eq:induction2}
|H_{k}(G_{k-1})| -|L_{k}(G_{k-1})|\geq h-(\ell-(k-1)) \geq  2k+(k-1) =  3k -1.
\end{equation}
Lemma~\ref{lemma:4k-1} implies that $G_{k-1}$ contains $3k$ vertices and by Theorem~\ref{thm:2k+t}, $G_{r} = G_{k-1}$ contains $k$ disjoint cycles.  Finally if $r = \ell = \min \{ \ell, k - 1 \}$, then $L_{r} = \emptyset$ and $|H_{k}(G_{r})| \geq 2k$.  Corollary~\ref{cor:K-K-Y} implies $G_{r} = G_{\ell}$ contains $k$ disjoint cycles.  Therefore, in any case $G_{r}$ contains $k$ disjoint cycles and by Lemma~\ref{lemma:induction} \emph{(7)}, $G$ contains $k$ disjoint cycles as well.

\section{Concluding remarks}\label{sec:conclusion}

\begin{remark} As mentioned earlier, there are graphs $G$ with $|G|\geq 3k$ and $|H_k(G)| - |L_k(G)| \geq 2k$ that
have no $k$ disjoint cycles, but all examples that we know have rather few vertices. The largest such graph $G$ that we can construct has  $4k$ vertices and is obtained as follows.

Let $F$ be a copy of $K_{3k-1}$. Choose $W\subset V(F)$ with $|W|=k$ and  delete all edges between the vertices in $W$. Then
add $k+1$ new vertices $x_0,x_1,\ldots,x_k$, and make $x_0$ adjacent to $x_1,\ldots,x_k$ and all vertices in $W$. In other words, let $(K_{2k-1}+K_1)\vee \overline K_k$ be the $2$-core of $G$, and complete the construction  by adding $k$ leaves adjacent to $x_0$, where $V(K_1)=\{x_0\}$.

Now $L_k(G)=\{x_1,\ldots,x_k\}$, and $H_k(G)=V(G)\setminus L_k(G)$. This graph has no $k$ disjoint cycles, 
since its $2$-core has $3k$ vertices, and $x_0$ does not belong to any triangle.

Is it true that every graph $G$ with $|G|\geq 4k+1$ and $|H_k(G)| - |L_k(G)| \geq 2k$ has $k$ disjoint cycles? 
\end{remark}

\begin{remark}  
Lemma~\ref{lemma:2k-1} suggests that considering $|V_{\geq 2k+1}(G)| - | V_{\leq 2k-1}(G)|$ instead of $|H_{k}(G)|  - |L_{k}(G)|$ 
may result in different bounds providing the existence of $k$ disjoint cycles.  
It could be  that the claim of Lemma~\ref{lemma:2k-1} holds not only for triangle-free graphs.  That is,  it could be that 
\emph{for any non-empty graph $G$ with $|V_{\geq 2k+1}(G)| - | V_{\leq 2k-1}(G)| \geq 2k - 2$, $G$ contains $k$ disjoint cycles.}  
This is trivially true for $k = 1$.
\end{remark}

\bibliographystyle{abbrv}
\bibliography{DiracErdos}

\end{document}